\documentclass[11pt, reqno]{amsart}

\usepackage[margin=1in]{geometry}
\usepackage{amsmath,amssymb,amsthm,mathtools}
\usepackage{hyperref}
\hypersetup{
    colorlinks=true,
    linkcolor=blue,
    filecolor=blue,
    citecolor=blue,
    urlcolor=cyan}

\newtheorem{theorem}{Theorem}[section]
\newtheorem{remark}[theorem]{Remark}
\newtheorem{corollary}[theorem]{Corollary}
\newtheorem{proposition}[theorem]{Proposition}
\newtheorem{lemma}[theorem]{Lemma}
\theoremstyle{definition}
\newtheorem{definition}[theorem]{Definition}

\newcommand{\R}{\mathbb{R}}
\newcommand{\E}{\mathbb{E}}
\newcommand{\Pp}{\mathbb{P}}
\newcommand{\1}{\mathbf{1}}
\newcommand{\norm}[1]{\left\lVert #1\right\rVert}
\newcommand{\dd}{\,\mathrm{d}}

\newcommand{\eps}{\varepsilon}

\title{Online Beck--Fiala Down to Logarithmic Sparsity}
\author{Dylan J. Altschuler and Konstantin Tikhomirov}

\address{Department of Mathematics, The University of Texas at Austin}
\email{dylan.altschuler@austin.utexas.edu}

\address{Department of Mathematical Sciences, Carnegie Mellon University}
\email{ktikhomi@andrew.cmu.edu}

\date{August 5, 2026}

\begin{document}

\begin{abstract}
The celebrated (offline)
Beck--Fiala conjecture asserts that every matrix $A\in\{0,1\}^{n\times T}$ with at most $d$ nonzero entries in each column has discrepancy $O(\sqrt d)$. In this note, we consider the {\it online} version of the conjecture.
For every fixed $\eta>0$, we give a randomized online algorithm that achieves $O_\eta(\sqrt d)$ prefix discrepancy for every fixed sequence of $d$-sparse vectors $a_t\in[-1,1]^n$, except with probability at most
\[
 C_\eta T\exp\left[-\frac{c_\eta d}{\log^{2+\eta}(ed)}\right].
\]
In particular, this gives an optimal $O(\sqrt d)$ discrepancy upper bound for
$d\geq C_\eta(\log T)(\log\log T)^{2+\eta}$. This range of sparsity is optimal up to $\log\log T$ factors. The result also resolves in strong form the {\it online Spencer problem}, corresponding to the case $d = n$. 

The algorithm is based on a compactly supported Metropolis walk, constructed
by combining ideas from several recent works on online vector balancing.
The proof was generated in conversation with ChatGPT 5.6 Pro; the authors
provided high-level guidance in several rounds of prompting, followed by
manual checking and rewriting of the proof.
\end{abstract}

\maketitle

\section{Introduction}

The \textit{offline} Beck--Fiala conjecture asserts that every binary matrix $A \in \{0,1\}^{n \times T}$ with $d$--sparse columns has discrepancy $O(\sqrt{d})$; crucially, the conjectured bound does not depend on the dimensions $n$ and $T$. The classical Beck--Fiala theorem \cite{BeckFiala1981} gives an upper bound of order $d$, while a celebrated result of Banaszczyk~\cite{Banaszczyk1998} gives the bound 
$O(\sqrt{d\log T})$, which provides an improvement in the superlogarithmic regime $d \gg \log T$. Recent breakthrough results of Bansal and
Jiang \cite{BansalJiang2026} 
 establish the conjectured $O(\sqrt d)$ bound for
$d \ge \log(T)^2$; additionally, for any sparsity, they show the uniform bound of
\[
    \widetilde O(\sqrt d+\sqrt{\log T})\,,
\]
where polynomial factors of $\log\log T$ are suppressed in the $\widetilde O$ notation. 
The (offline) Beck--Fiala and the stronger Koml\'os conjecture
were successfully resolved in full the past few months \cite{komlos, komlos-exposition}.

We consider the closely related problem of algorithmic \textit{online} discrepancy minimization. A sequence of hidden vectors $a_1,\dots,a_T$ are fixed ahead of time by an oblivious adversary. Then, the vectors are sequentially revealed to the algorithm which must immediately and irrevocably choose signs $\eps_1,\dots, \eps_T$ in order to minimize $\|\sum_{i \le T} \eps_i a_i\|_\infty$. Here, ``online'' refers to the algorithmic restriction that $\eps_t$ is chosen just as a function of $a_1,\dots,a_t$, while ``oblivious'' refers to the fact that the vectors $a_1,\dots,a_T$, although hidden from the algorithm, are fixed ahead of time and hence cannot be adversarially adapted to the algorithm's choices. Finally, an algorithm which further guarantees bounds on 
\[
    \max_{t \le T} \left\|\sum_{1 \le s \le t} a_s \eps_s\right\|_\infty 
\]
is said to minimize ``prefix'' discrepancy.
% \begin{remark}[Online prefix discrepancy]
%     The discrepancy minimization literature is inconsistent on the convention of whether online algorithms are required to minimize prefix discrepancy or just terminal discrepancy. Further, there is a substantial body of work on a prefix discrepancy minimization in offline settings. There will be no ambiguity in the present article, as the algorithm considered is both online and minimizes prefix discrepancy. 
% \end{remark}

For online vector balancing, Kulkarni, Reis, and Rothvoss~\cite{KRR2024}
proved a prefix discrepancy bound of $O(\sqrt{\log T})$ for vectors of
Euclidean norm at most one.  Aden-Ali~\cite{AdenAli2026} subsequently gave
another algorithm, which strongly
influenced the present paper, achieving the same bound.
Since a $d$-sparse
vector with entries in $[-1,1]$ has Euclidean norm at most $\sqrt d$,
these results give the online prefix bound
\begin{equation}\label{eq:direct-online-bound}
  O\!\left(\sqrt{d \log T}\right)\,.
\end{equation}
Prior to the present work, this was, to our knowledge, the best available bound for the online Beck--Fiala problem. In particular, it was even an open problem whether online discrepancy grows at the same $O(\sqrt{n})$ asymptotic rate as offline discrepancy in the classical ``Spencer'' setting of square matrices with bounded entries \cite{spencer}.

\subsection{Results}
The following theorem gives the conjectured $O(\sqrt d)$ discrepancy
down to nearly logarithmic sparsity.  The signs are chosen online, and
the bound holds for every prefix.

\begin{theorem}[Online Beck--Fiala]\label{thm:main}
For every fixed $\eta>0$ there are constants $c_\eta,C_\eta>0$ with the
following property.  Let $d\geq2$.  There is a randomized online
algorithm such that, for every $T$ and every fixed sequence
$a_1,\ldots,a_T\in[-1,1]^n$ with at most $d$ nonzero coordinates in each
vector,
\begin{equation}\label{eq:main-tail}
 \Pp\left\{
   \max_{t\leq T}\Big\|\sum_{s\leq t}\varepsilon_sa_s\Big\|_\infty
     >C_\eta\sqrt d
 \right\}
 \leq C_\eta T
 \exp\left[-\frac{c_\eta d}{\log^{2+\eta}(ed)}\right].
\end{equation}
Consequently, there are constants $C,C',T_0>0$, depending only on $\eta$,
such that, whenever $T\geq T_0$ and
\begin{equation}\label{eq:main-threshold}
 d\geq C(\log T)(\log\log T)^{2+\eta},
\end{equation}
the prefix discrepancy is at most $C'\sqrt d$ with probability at least
$1-T^{-1}$. 
\end{theorem}

\begin{remark}
When the vectors $(a_t)_{t \in T}$ are presented in sparse form, the above algorithm uses
\[
 O_\eta\left(n+T+\sum_{t\leq T}|\operatorname{supp}a_t|\right)
\]
expected operations.  The three terms account for initialization,
processing the $T$ arrivals, and updating the nonzero coordinates,
respectively.
\end{remark}

%We collect two immediate consequences in relation to existing conjectures.

%\begin{corollary}[Offline Beck--Fiala conjecture]
%For every fixed $\eta>0$ and $T\geq T_0(\eta)$, let
%$A \in \{0,1\}^{n \times T}$ have at most $d$
%nonzero entries in each column, where $d$ satisfies
%\eqref{eq:main-threshold}.  Then
%\[
%    \mathrm{disc}(A) = O_\eta\big(\sqrt{d}\big)\,.
%\]
%\end{corollary}

%Next, 
It follows from taking $d = n$ in Theorem \ref{thm:main} (and adjusting the constants) that the online prefix discrepancy in the Spencer setting is bounded by $O(\sqrt{n})$. This answers, for instance, Conjecture 1 of \cite{KRR2024} affirmatively and in strong form. Previously, such bounds were only known in the ``average case'' setting where the incoming vectors are sampled uniformly at random from $\{-1,+1\}^n$ \cite{BS2019}.

\begin{corollary}[Online prefix Spencer]
\label{cor:online-spencer}
For every $\eta>0$ there are constants $c_\eta,C_\eta>0$ and
$n_0(\eta)$ such that the following holds.  If $n\geq n_0(\eta)$ and
\[
  T\leq
  \exp\left(\frac{c_\eta n}{\log^{2+\eta}(en)}\right),
\]
then there is an efficient, randomized online algorithm such that, for every fixed
sequence $a_1,\ldots,a_T\in[-1,1]^n$,
\[
 \Pp\left\{
  \max_{t\leq T}
  \left\|\sum_{s=1}^t\varepsilon_sa_s\right\|_\infty
  \leq C_\eta\sqrt n
 \right\}
 \geq
 1-\exp\left[-\frac{c_\eta n}{\log^{2+\eta}(en)}\right].
\]
In particular, taking $T=n$ gives a positive answer to
Conjecture~1 of \cite{KRR2024}, with the stronger guarantee holding
simultaneously for every prefix.
\end{corollary}

Finally, we observe that the range of $d$ considered in \eqref{eq:main-threshold} is essentially optimal.

\begin{remark}[Kulkarni--Reis--Rothvoss \cite{KRR2024}, Theorem 4] There is a strategy for an oblivious adversary that produces a sequence of vectors $a_1,\dots,a_T \in [-1,1]^2$ such that the prefix discrepancy incurred by any online algorithm is, with probability at least $1-2^{-T^{\Omega(1)}}$, at least $\Omega(\sqrt{\log T})$.   
Since these vectors are $2$-sparse, no online
algorithm can guarantee $O(\sqrt d)$ prefix discrepancy throughout the
sublogarithmic regime $d=o(\log T)$, even against an oblivious adversary.
\end{remark}

\subsection{Related work}
The interested reader is referred to the survey \cite{BansalICM} 
for general background on the rich field of discrepancy theory, both online and offline. We now introduce some specific algorithmic developments in the study of the online Koml\'os problem that are closely related to the present article. 

Motivated by several problems in discrepancy theory, Alweiss, Liu, and Sawhney showed that incoming vectors in the online Koml\'os setting could be balanced in such a way to obtain strong control on the growth of the sub-Gaussian norm of prefix sums \cite{AlweissLiuSawhney2021}. A work of Liu, Sah, and Sawhney later showed that optimal bounds over this growth can be obtained if one allows so-called partial colorings: signings $\eps = (\eps_1,\dots,\eps_T) \in \{-1,0,1\}^T$ with $\eps_t \in \{-1,1\}$ for most $t$ \cite{LiuSahSawhney2022}. The aforementioned result of Kulkarni, Reis, and Rothvoss \cite{KRR2024} showed that optimal bounds can be obtained in the original setting of $\eps_i \in \{-1,+1\}$. 

Subsequently, a work of Smirnov and Vershynin \cite{SV2026} took a similar approach to \cite{LiuSahSawhney2022} and constructed a random walk that can be kept inside of a compact body by rejecting or discarding a controlled fraction of random steps. Finally, a result of Aden-Ali \cite{AdenAli2026} shows how to obtain the sub-Gaussian walk of Liu, Sah, and Sawhney using bona fide signings $\eps_i \in \{-1,1\}$, yielding an optimal algorithm in the online Koml\'os setting matching the non-constructive bounds of \cite{KRR2024}. 

%\begin{remark}[Authors' note: October 2026]
    Several months after this article first appeared as a preprint, the offline Koml\'os and Beck--Fiala conjectures were solved \cite{komlos, komlos-exposition}. While the present article was, at some point in time, incidentally the best known result for the offline Beck--Fiala conjecture, the main contribution is the {\it near-resolution of the online setting}, which is independent of the aforementioned progress in the offline settings. 
%\end{remark}

\subsection{AI Usage} In view of the above works, the authors prompted ChatGPT 5.6 Pro to solve the online Spencer problem with the outline of applying the coupling of \cite{AdenAli2026} to the bounded random walk of \cite{SV2026}. A complete proof for $O(\sqrt{n})$ online prefix discrepancy in the Spencer setting, with time horizon $T = n$, was produced in a single shot. The choice of the bump density in Section~\ref{sec:bump} was generated by GPT rather than from our initial proof strategy. Subsequent rounds of prompting extended the proof to the online Beck--Fiala setting treated here. The authors manually verified and rewrote the proofs, and then organized the results into the present manuscript.

\section{The compact triplet framework}\label{sec:framework}
We begin by introducing the framework of \cite{AdenAli2026},
which plays a critical role in the argument.

\subsection{A compactly supported walk}\label{subsec: compact walk}

Let $\rho$ be a positive probability density on $(-L,L)$,
to be defined later, extended to be zero
outside the interval. Define
\[
  \mu:=\rho^{\otimes n},
  \qquad
  \mu(x):=\prod_{i=1}^n\rho(x_i).
\]
It is convenient to parameterize the marginal $\rho$ in terms of a potential $u$:
\[
 \rho(x)=\frac1{LZ}e^{-u(x/L)},
 \qquad
 Z=\int_{-1}^1 e^{-u(y)}\dd y,
\]
where $u(y)=+\infty$ for $|y|\geq1$.  Thus
\[
 \mu(x)=(LZ)^{-n}\exp\left[-\sum_{i=1}^n u(x_i/L)\right].
\]
For any
$x\in (-L,L)^n$ and arbitrary $v\in\R^n$, set
\begin{align}
 p_v^+(x)&:=\frac13\min\left\{1,\frac{\mu(x+v)}{\mu(x)}\right\},
 \label{eq:p-plus}\\
 p_v^-(x)&:=\frac13\min\left\{1,\frac{\mu(x-v)}{\mu(x)}\right\},
 \label{eq:p-minus}\\
 p_v^0(x)&:=1-p_v^+(x)-p_v^-(x).
 \label{eq:p-zero}
\end{align}
Given a non-random sequence of vectors $v_1,v_2,\dots$ in $\R^n$,
we consider the $n$--dimensional random walk 
$
X_0,X_1,\dots,
$ where $X_0\sim\mu$ and $X_{t}$
is defined through 
$$
X_t=X_{t-1}+\delta_t\,v_t,
$$
with $\delta_t$ sampled from $\{-1,0,1\}$ via
$$
\Pp\{\delta_t=s\;|\,X_{t-1}\big\}=p_{v_t}^s(X_{t-1}).
$$
\begin{remark}
Note that the transition probabilities at step $t$
are determined by the input vectors $v_1,\dots,v_t$,
i.e., are computable in an online fashion. 
\end{remark}

\begin{lemma}[Stationarity]\label{lem:stationarity}
For every non-random sequence of vectors $v_1,v_2,\dots$ in $\R^n$,
$X_t\sim\mu$.
\end{lemma}
\begin{proof}
We argue by induction. Suppose $X_{t-1}\sim\mu$ and write $v=v_t$. Recall the convention that $\mu$ is zero outside the cube. Writing $\mu_t$ for the density of $X_t$, we then have for $y\in(-L,L)^n$,
\begin{align*}
    \mu_t(y) &= \mu(y)p_v^0(y) +\frac{1}{3}\min\{\mu(y-v),\mu(y)\} +\frac{1}{3}\min\{\mu(y+v),\mu(y)\} \\
    &= \mu(y)\bigl(p_v^0(y)+p_v^-(y)+p_v^+(y)\bigr)=\mu(y).
\end{align*}
Thus $X_t\sim\mu$. As the base case $X_0\sim\mu$ holds by construction, the induction is complete.
\end{proof}

The key property enabling analysis of the walk is the following:

\begin{definition}[Balanced state]
A vector $x\in(-L,L)^n$ is \emph{$v$-balanced} if
\begin{equation}\label{eq:balanced}
   p_v^+(x)+p_v^-(x)\geq\frac13.
\end{equation}
\end{definition}

\begin{definition}[Curvature]
For $x$ satisfying both $x + v\in(-L,L)^n$ and $x - v\in(-L,L)^n$, define
\begin{align*}
 \Delta_\pm(x,v)
 &:=\sum_{i=1}^n\left[u\left(\frac{x_i\pm v_i}{L}\right)
                         -u\left(\frac{x_i}{L}\right)\right]
   =-\log\frac{\mu(x\pm v)}{\mu(x)},\\
 D_v(x)&:=\Delta_+(x,v)+\Delta_-(x,v).
\end{align*}
\end{definition}

We will utilize the following convenient sufficient condition for balance. 

\begin{lemma}[Curvature criterion]\label{lem:curvature}
If $x+v\in(-L,L)^n$, $x-v\in(-L,L)^n$, and $D_v(x)\leq2\log2$, then $x$ is $v$-balanced.
\end{lemma}
\begin{proof}
If $\Delta_+(x,v)\leq0$, then $\mu(x+v)\geq \mu(x)$ and
$p_v^+(x)=1/3$; the same applies to the other direction.
If both increments are positive, convexity of $s\mapsto e^{-s}$ gives
\[
 p_v^+(x)+p_v^-(x)
 =\frac13\left(e^{-\Delta_+(x,v)}+e^{-\Delta_-(x,v)}\right)
 \geq\frac23e^{-D_v(x)/2}\geq\frac13.
\]
\end{proof}

\subsection{The three-way coupling}

We use the following finite coupling lemma, in the form established by
Aden-Ali~\cite{AdenAli2026}. 

\begin{lemma}[Three-way coupling; \cite{AdenAli2026}]\label{lem:coupling}
For $j\in\{1,2,3\}$, let $a_j,b_j$ satisfy
\[
  0\leq a_j,b_j\leq\frac13,
  \qquad
  a_j+b_j\geq\frac13.
\]
There is a coupling $(\delta_1,\delta_2,\delta_3)$
of random variables in $\{-1,0,1\}$ such that
\[
 \Pp\{\delta_j=1\}=a_j,
 \qquad
 \Pp\{\delta_j=-1\}=b_j,
\]
for every $j$, and
\[
   \delta_1+\delta_2+\delta_3\in\{-1,1\}
   \qquad\text{almost surely}.
\]
\end{lemma}

\subsection{The coupled Metropolis walk}
For each choice of invariant density $\rho$, we are able to construct a well-behaved random walk in an online fashion. We now extract the precise corresponding algorithmic result for arbitrary $\rho$, which will later inform our actual choices for $\rho$.

\begin{proposition}\label{prop:generic}
Let $\mathcal V\subseteq\R^n$ be a family of column vectors.  Suppose a
product law $\mu$ supported on $(-L,L)^n$ satisfies
\begin{equation}\label{eq:generic-balance}
  \sup_{v\in\mathcal V}
  \Pp_{X\sim\mu}\{X\text{ is not $v$-balanced}\}
  \leq\theta.
\end{equation}
Then there is a randomized online vector balancing algorithm
such that for every fixed sequence $v_1,\ldots,v_T\in\mathcal V$,
\begin{equation}\label{eq:generic-conclusion}
 \Pp\left\{
  \max_{k\leq T}\norm{\sum_{t \le k}\varepsilon_tv_t}_\infty>6L
 \right\}
 \leq3T\theta.
\end{equation}
\end{proposition}

\begin{proof}
We construct three auxiliary coupled random walks 
$$
(X_{t,j})_t,\quad j=1,2,3,
$$
with each $(X_{t,j})_t$
defined as in Subsection~\ref{subsec: compact walk};
namely, the three initial states are sampled independently from $\mu$,
with transition probabilities 
computed via \eqref{eq:p-plus}, 
\eqref{eq:p-minus},
\eqref{eq:p-zero}. 
The coupling is produced as follows.
Let $\mathcal F_{t-1}$ denote the algebra generated by $(X_{s,j})_{s\le t- 1,\, j \in [3]}$.
At step $t\geq 1$, compute the six
probabilities $p_{v_t}^\pm(X_{t-1,j})$.

If all three previous states $X_{t-1,j}$ are $v_t$-balanced, apply
Lemma~\ref{lem:coupling} to sample the coupled variables $\delta_{t,j}$
with the marginals
$$
\Pp\{\delta_{t,j}=s\mid\mathcal F_{t-1}\}
=p_{v_t}^s(X_{t-1,j}),\quad s=\pm,0,
$$
so that
\[
  \varepsilon_t:=\delta_{t,1}+\delta_{t,2}+\delta_{t,3}\in\{-1,1\},
\]
and set
$$
X_{t,j}=X_{t-1,j}+\delta_{t,j}v_t,\quad j=1,2,3.
$$
Otherwise, if the previous three states are not all balanced, sample the three increments independently according to their
one-walk marginals, update the three states by the same rule, and output
an arbitrary sign $\varepsilon_t\in\{-1,1\}$.  Repeat this procedure at
every step.

\medskip

At every step, each increment has the same distribution as one of the ``marginal'' individual walks $(X_{s,j})_s$ even after
conditioning on the entire triple of current states.  In particular, regardless of the non-random vectors
$(v_s)_{s\le t}$, every $X_{t,j}$ is distributed according to $\mu$.
A union bound using \eqref{eq:generic-balance} shows that all states are balanced up to time $T$ simultaneously, with probability at least $1-3T\theta$.  On this event,
the coupling is used at every step (so that the walk corresponds to a genuine signing), and for every prefix $k$,
\begin{align*}
 \sum_{t=1}^k\varepsilon_tv_t
 =\sum_{j=1}^3\sum_{t=1}^k\delta_{t,j}v_t
 =\sum_{j=1}^3(X_{k,j}-X_{0,j}).
\end{align*}
Every state remains in $(-L,L)^n$, so the last expression has infinity norm at
most $6L$, simultaneously for all $k$.
\end{proof}

\begin{remark}[Obliviousness]\label{rem:oblivious}
The proof uses the obliviousness of the adversary: $v_t$ is fixed independently of the current private state of the algorithm.
For an adaptive adversary, conditioning on earlier output signs can bias the
law of $X_{t-1,j}$, and \eqref{eq:generic-balance} cannot be applied in this form.
\end{remark}

\section{Construction of the invariant measure}\label{sec:bump}
The potential that will be used is the following. For fixed $\eta>0$, define 
\[
u(y):=
\begin{cases}
\displaystyle \exp\!\left((1-y^2)^{-2/\eta}\right), & |y|<1,\\
+\infty, & |y|\geq1.
\end{cases}
\]
In particular, the corresponding probability density $\rho(x)$ has a double exponential structure:
\[
\rho(x)=\frac{1}{LZ}\exp\!\left(-u(x/L)\right),
\qquad
Z:=\int_{-1}^{1}e^{-u(y)}\,dy,
\]
where the parameter $L$ is chosen shortly to be an implicit, sufficiently large constant.

\subsection{Choice of the potential $u$.} We now give a heuristic derivation of the above choice. Begin by considering how the balance condition relates to the updates given by the choices of signs $\delta_t$.

Recall  $\Delta_\pm(x,v) =-\log\frac{\mu(x\pm v)}{\mu(x)}$, $D_v(x) :=\Delta_+(x,v)+\Delta_-(x,v)$, and that Lemma \ref{lem:curvature} yields that $|D_v(x)| \le 2\log 2$ implies balance. We focus on showing that $|D_v(x)|$ is bounded above by a constant for $L=1$; by chain rule, taking $L$ sufficiently large then gives $|D_v(x)| \le 2\log 2$. Taylor expansion suggests for an incoming vector $v$,
\[
 \Delta_\pm(x,v) =\sum_{i=1}^n \bigl(u(x_i\pm v_i)-u(x_i)\bigr) \approx \pm\sum_{i=1}^n v_i u'(x_i) +\frac12\sum_{i=1}^n v_i^2 u''(x_i).
\]
and hence
\[
    D_v(x)\approx\sum_{i=1}^n v_i^2 u''(x_i).
\]
Assume that
the arriving vector $v$ is normalized so that
$\|v\|_\infty \le 1/\sqrt{d}$. The key requirement is that the (unweighted) sum of $u''(x_i)$ over the $d$ non-zero coordinates of $v$ is bounded above by $Kd$, for some constant $K$, except with some small failure probability proportional to $1/T$ in order to use
Proposition~\ref{prop:generic}. Thus, whatever potential $u$ one picks must enforce a strong tail bound on $u''(X)$ for $X \sim \rho$. In order to allow $d$ to be as small as $\log T$, up to powers of
$\log\log T$, one therefore seeks a one-step failure bound of the form
$\exp[-cd/\log^p(ed)]$, with $p$ as small as possible.

To construct such a potential, let $Y$ have density $Z^{-1}e^{-u}$
on $(-1,1)$. The starting observation is that the potential itself
already has an exponential tail:
\begin{equation}\label{eq:bump-potential-tail}
 \Pp\{u(Y)>r\}
 =\frac1Z\int_{\{y\in(-1,1):\,u(y)>r\}}e^{-u(y)}\,dy
 \leq \frac{2}{Z}e^{-r}.
\end{equation}
Thus it suffices to arrange that
\begin{equation} \label{eq:u'' requirement}
    u''(y)\leq C u(y)(\log u(y))^p
\end{equation}
for large $u(y)$, with $u''$ bounded elsewhere. Indeed, under this
bound, curvature exceeding $r$ requires a potential of at least
$c r/\log^p(er)$, and \eqref{eq:bump-potential-tail} gives precisely
the desired tail.

Why allow the logarithmic factor? Ideally one would take $p=0$,
giving an exponential curvature tail and sparsity of order $\log T$.
However, we also want $u(y)\to\infty$ as $y\uparrow1$, so that the
density vanishes at the boundary. It turns out that $p > 2$ is necessary for this blowup. Indeed, suppose that $p\geq0$ and that the potential satisfies \eqref{eq:u'' requirement}
for all $y$ sufficiently close to $1$. Fix $y_0\in(0,1)$ so that
this inequality holds and $u(y)\geq e$ throughout $[y_0,1)$.
Since $u$ is even and convex, $u'\geq0$ on this interval. Multiplying \eqref{eq:u'' requirement} by $2u'$ and then integrating from $y_0$ to $y$ gives
\begin{align*}
 (u'(y))^2-(u'(y_0))^2
 &=2\int_{y_0}^y u''(z)u'(z)\,dz \leq 2C\int_{u(y_0)}^{u(y)} t(\log t)^p\,dt
 \leq C u(y)^2(\log u(y))^p,
\end{align*}
where one uses substitution $t=u(z)$ and then the bound $\log t\leq\log u(y)$. Hence
\[
 \frac{u'(y)}{u(y)(\log u(y))^{p/2}}\leq C'.
\]
Integrating this last inequality from $y$ to $1$, and using
$u(z)\to\infty$ as $z\uparrow1$, gives
\[
 \int_{u(y)}^\infty\frac{dt}{t(\log t)^{p/2}}
 =\int_y^1\frac{u'(z)}{u(z)(\log u(z))^{p/2}}\,dz
 \leq C'(1-y).
\]
The left-hand side is finite precisely
when $p>2$, hence we take $p=2+\eta$, for any fixed
$\eta>0$. This is the source of the power $2+\eta$ on $\log\log T$
in the sparsity bound.

Next, consider a potential $u$ satisfying the preceding necessary estimate: assume for some $\eta > 0$ and all $y$ near one,
\[
 u'(y)\asymp u(y)(\log u(y))^{1+\eta/2}. 
\]
The double exponential structure of the density emerges here after recognizing that this equation is a logarithmic derivative. Considering the parameterization $A=\log u$, one obtains:
\[
    A'\asymp A^{1+\eta/2}.
\]
Noting that $\lim_{b\to 1^-} A(b)^{-\eta/2} = (+\infty)^{-\eta/2} = 0$, one can solve this equation by integrating:
\[
 \frac{2}{\eta}A(y)^{-\eta/2}
 = \lim_{b\to 1^-}\frac{2}{\eta}(A(y)^{-\eta/2} - A(b)^{-\eta/2})
 = \int_y^1\frac{A'(t)}{A(t)^{1+\eta/2}}\,dt
 \asymp 1-y.
\]
That is, $A$ must satisfy an inverse power law in $(1-y)$ for $y$ near one. It remains to handle $y$ near $-1$. As $1-y^2$ is symmetric and satisfies $1-y^2 \asymp 1 -y$ for $y$ near one, it is natural to choose
\begin{equation}\label{eq:bump-potential-main}
 A(y):=(1-y^2)^{-2/\eta},
 \qquad
 u(y):=e^{A(y)}, \qquad |y|<1,
\end{equation}
and set $u(y)=+\infty$ for $|y|\geq1$. This completes the heuristic derivation of $u$. \\

\subsection{Bounding the curvature of $u$}
We now make parts of the above heuristic derivation precise. As above, we take $L=1$ throughout this subsection so that
$\rho(y)=Z^{-1}e^{-u(y)}$ on $(-1,1)$. 

First, as a sanity check, our chosen form for the potential satisfies \eqref{eq:u'' requirement}. Indeed, direct differentiation gives
\[
 A'(y)=\frac{4y}{\eta}A(y)^{1+\eta/2},
 \qquad
 0\leq A''(y)\leq C_\eta A(y)^{1+\eta}.
\]
Since $A\geq1$, we obtain
\begin{equation}\label{eq:bump-second-main}
 0\leq u''(y)
 =u(y)\bigl((A'(y))^2+A''(y)\bigr)
 \leq C_\eta u(y)(\log u(y))^{2+\eta}.
\end{equation}
Next, and more importantly, we need to make the error in the above Taylor expansion of $u$ precise. Note that if $x_i \pm v_i\in(-1,1)$, Taylor's remainder theorem gives
\[
 u(x_i+v_i)+u(x_i-v_i)-2u(x_i)
 \leq v_i^2\sup_{|z-x_i|\leq |v_i|}u''(z).
\]
Thus, what is really needed is control on $u''$ between $x_i - v_i$ and $x_i + v_i$ (as opposed to our heuristic derivation which only used control of $u''$ at $x_i$) as well as conditions for both endpoints $x_i \pm v_i$ to be in $(-1,1)$. The following lemma shows that controlling the right-hand side of \eqref{eq:bump-second-main} provides the desired control. Here, $y$ and $h$ correspond to $x_i$ and $v_i$ from our derivation, respectively.

\begin{lemma}[Curvature of $u$]\label{lem:coordinate-cost} Let $Y \sim \rho$.
There are constants $M_\eta\geq1$ and $c_\eta,C_\eta > 0$, each depending only on $\eta$, such that
\begin{equation}\label{eq:coordinate-cost}
 W(y):=M_\eta\, u(y)\,A(y)^{2+\eta}
\end{equation}
has the following properties.  First, for every $x\geq1$,
\begin{equation}\label{eq:coordinate-cost-tail}
 \Pp\{W(Y)>x\}
 \leq C_\eta\exp\left[-\frac{c_\eta x}{\log^{2+\eta}(ex)}\right].
\end{equation}
Second, if $y\in(-1,1)$ and $s^2W(y)\leq1$, then $y\pm s\in(-1,1)$ and
\begin{equation}\label{eq:local-coordinate-cost}
 u(y+s)+u(y-s)-2u(y)\leq s^2W(y).
\end{equation}
\end{lemma}

\begin{proof}
For any $M_\eta\geq1$, the condition
\[
 u(y)\leq\frac{x}{M_\eta\log^{2+\eta}(ex)}
\]
implies $A(y)=\log u(y)\leq\log x$ and hence $W(y)\leq x$.
Consequently, \eqref{eq:bump-potential-tail} gives
\[
 \Pp\{W(Y)>x\}
 \leq\Pp\left\{u(Y)>\frac{x}{M_\eta\log^{2+\eta}(ex)}\right\}
 \leq C_\eta\exp\left[-\frac{x}{M_\eta\log^{2+\eta}(ex)}\right].
\]

We now choose $M_\eta$ large enough for the second assertion.
Since $u,A\geq1$,
\[
 W(y)\geq M_\eta A(y)^\eta
 =\frac{M_\eta}{(1-y^2)^2}.
\]
Thus $s^2W(y)\leq1$ implies
$|s|\leq(1-y^2)/\sqrt{M_\eta}$.  For $M_\eta\geq16$, this keeps
$y\pm s$ inside $(-1,1)$ and ensures that
\[
 1-z^2\geq\frac{1-y^2}{2}
 \qquad\text{whenever }|z-y|\leq|s|.
\]
The derivative bound for $A$ and the mean value theorem then give
\[
 |A(z)-A(y)|
 \leq C_\eta|s|A(y)^{1+\eta/2}
 \leq\frac{C_\eta}{\sqrt{M_\eta}}
 \leq1,
\]
In particular, since $A(y)\geq1$, this gives $A(z)\leq A(y)+1\leq2A(y)$ and
\[
 u(z)=e^{A(z)}\leq e^{A(y)+1}=e\,u(y).
\]
By
\eqref{eq:bump-second-main}, further increasing $M_\eta$ so that $M_\eta\geq e\,2^{2+\eta}C_\eta$, we obtain
\[
 \sup_{|z-y|\leq |s|}u''(z)\leq \sup_{|z-y|\leq |s|} C_\eta (e\, u(y)) (2 A(y))^{2+\eta}
 \leq M_\eta u(y)A(y)^{2+\eta}=W(y).
\]
Hence, Taylor's theorem with the standard remainder estimate proves
\eqref{eq:local-coordinate-cost}.
\end{proof}

\section{Combining the coordinates}\label{sec:concentration}

We now bound the sum of at most $d$ independent copies of $W(Y)$ from Lemma~\ref{lem:coordinate-cost}. We first exclude values above $d$, then apply an exponential estimate to the truncated sum.

\begin{lemma}[Averaging the coordinate costs]\label{lem:sum-tail}
There are $K_\eta\geq1$ and $C_\eta,c_\eta>0$ such that, for every
$d\geq2$, every integer $1\leq r\leq d$, and independent copies
$W_1,\ldots,W_r$ of $W(Y)$,
\begin{equation}\label{eq:sum-tail-conclusion}
 \Pp\left\{\sum_{i=1}^r W_i>K_\eta d\right\}
 \leq C_\eta\exp\left[-\frac{c_\eta d}{\log^{2+\eta}(ed)}\right].
\end{equation}
\end{lemma}

\begin{proof}
There are at most $d$ variables, so \eqref{eq:coordinate-cost-tail} and
a union bound give
\begin{equation}\label{eq:sum-truncation}
 \Pp\{\max_i W_i>d\}
 \leq C_\eta d\exp\left[-\frac{c_\eta d}{\log^{2+\eta}(ed)}\right]
 \leq C_\eta'\exp\left[-\frac{c_\eta' d}{\log^{2+\eta}(ed)}\right].
\end{equation}
Put $\widehat W_i=\min\{W_i,d\}$ and choose
\[
 \lambda=\frac{c_\eta}{2\log^{2+\eta}(ed)},
\]
where $c_\eta$ is the constant in \eqref{eq:coordinate-cost-tail}. Our choice of $\lambda$ gives
\begin{equation}\label{eq:lambda x bound}
     \lambda x\leq\frac{c_\eta x}{2\log^{2+\eta}(ex)}
 \qquad(1\leq x\leq d).
\end{equation}

Since $\widehat W_i=\min\{W_i,d\}$, we have
\[
 e^{\lambda\widehat W_i}-1
 =\lambda\int_0^d e^{\lambda x}\1_{\{W_i>x\}}\dd x.
\]
Taking expectations, splitting the integral at $1$, and applying
\eqref{eq:coordinate-cost-tail} and \eqref{eq:lambda x bound} on $[1,d]$, we obtain
\begin{align*}
 \E e^{\lambda\widehat W_i}
 &=1+\lambda\int_0^d e^{\lambda x}\Pp\{W_i>x\}\dd x\\
 &\leq1+\lambda\left(e^{c_\eta/2}
      +C_\eta\int_1^\infty
       \exp\left[-\frac{c_\eta x}{2\log^{2+\eta}(ex)}\right]\dd x\right)
 \leq e^{B_\eta\lambda},
\end{align*}
where $B_\eta$ is a finite constant. For $x \in [0,1]$, we used here the trivial bound $e^{\lambda} \le e^{c_\eta/2}$. Independence, $r\leq d$, and
Markov's inequality, with $K_\eta=B_\eta+1$, now give
\[
 \Pp\left\{\sum_i\widehat W_i>K_\eta d\right\}
 \leq e^{-\lambda K_\eta d}\prod_{i=1}^r\E e^{\lambda\widehat W_i}
 \leq e^{-\lambda d}.
\]
Combining this with \eqref{eq:sum-truncation} proves the lemma.
\end{proof}

\subsection{Proof of the main theorem}

\begin{proof}[Proof of Theorem~\ref{thm:main}]
Let $K_\eta$ be the constant from Lemma~\ref{lem:sum-tail}, set
$L=\sqrt{K_\eta}$, and take $\rho$ to be the density of $LY$.
Normalize each arriving column as $v=a/\sqrt d$.  For a random state
$X=(LY_1,\ldots,LY_n)$ with independent copies $Y_i$ of $Y$, let
$I=\{i:v_i\ne0\}$ and $W_i=W(Y_i)$.  There are at most $d$ active
coordinates, so Lemma~\ref{lem:sum-tail} gives
\[
 \Pp\left\{\sum_{i\in I}W_i>K_\eta d\right\}
 \leq C_\eta\exp\left[-\frac{c_\eta d}{\log^{2+\eta}(ed)}\right].
\]
The empty sum is zero, so this also covers $v=0$.

Suppose $\sum_{i\in I}W_i\leq K_\eta d$.  Since $v_i^2\leq1/d$, each
active coordinate satisfies
\[
 \left(\frac{v_i}{L}\right)^2W_i
 \leq\frac{W_i}{L^2d}\leq\frac{K_\eta}{L^2}=1.
\]
Applying Lemma~\ref{lem:coordinate-cost} with $y=Y_i$ and $s=v_i/L$ shows that $X\pm v\in(-L,L)^n$. Summing \eqref{eq:local-coordinate-cost} over $i\in I$ gives
\[
 D_v(X)
 =\sum_{i\in I}\bigl(u(Y_i+v_i/L)+u(Y_i-v_i/L)-2u(Y_i)\bigr)
 \leq\frac1{L^2d}\sum_{i\in I}W_i
 \leq1<2\log2.
\]
Thus $X$ is $v$-balanced by Lemma~\ref{lem:curvature}.  The probability
that balance fails is at most the preceding tail bound, uniformly over
all normalized $d$-sparse columns.

By Proposition~\ref{prop:generic}, every normalized prefix has infinity norm at most $6L$, except on an event of probability at most
\[
 C_\eta T\exp\left[-\frac{c_\eta d}{\log^{2+\eta}(ed)}\right].
\]
Multiplying the discrepancy bound by $\sqrt d$ proves \eqref{eq:main-tail}.  Substituting
$d=C(\log T)(\log\log T)^{2+\eta}$ in this probability bound, and using
that $d/\log^{2+\eta}(ed)$ is eventually increasing, gives the
high probability conclusion for sufficiently large $C$ and $T$. \\
\end{proof}

\end{document}